\documentclass[10pt,double,reqno]{amsart}

\usepackage{amsmath,amssymb,amscd,color}

\usepackage{graphicx}
\usepackage{epsfig}
\usepackage{amsthm}

\usepackage{enumerate}
\usepackage{amsbsy}
\usepackage{amsfonts}
\newtheorem{theo}{Theorem}[section]
\newtheorem{defin}[theo]{Definition}
\newtheorem{prop}[theo]{Proposition}

\newtheorem{lemm}[theo]{Lemma}
\newtheorem{rem}[theo]{Remark}

\newcommand{\al}{\alpha}

\newcommand{\Ga}{\Gamma}

\newcommand{\om}{\omega}
\newcommand{\Om}{\Omega}

\newcommand{\De}{\Delta}

\newcommand{\pa}{\partial}

\newcommand{\R}{{\mathbb R}^n}

\newcommand{\Rn}{{\mathbb R}^{n-1}}

\newcommand{\na}{\nabla}

\begin{document}
\baselineskip=18pt

\title[]{Initial-Boundary Value Problem of the   Navier-Stokes system in the  half space}

\
\author{Tongkeun Chang}
\address{Department of Mathematics, Yonsei University \\
Seoul, 136-701, South Korea}
\email{chang7357@yonsei.ac.kr}

\author{Bum Ja Jin}
\address{Department of Mathematics, Mokpo National University, Muan-gun 534-729,  South Korea }
\email{bumjajin@hanmail.net}

\thanks{}

\begin{abstract}
In this  paper, we study the initial-boundary value problem of the Navier-Stokes system in the half space.
We prove the  unique solvability of the weak solution  on some  short time interval $(0,T)$ with the velocity in $ C^{\al,\frac{\al}{2}}(\R_+\times (0,T))$, $0<\al<1$,  when the given initial data is in $ C^\al(\R_+)$ and the given  boundary data is in $ C^{\al,\frac{\al}{2}}(\Rn\times (0,T))$. Our result generalizes the result in \cite{sol1}  considering nonhomogeneous Dirichlet boundary data.



\end{abstract}

\maketitle

\section{\bf Introduction}
\setcounter{equation}{0}

Let $\R_+ = \{ x \in \R \, | \, x_n > 0 \}$, $n\geq 2$ and $0 < T < \infty$.
In this paper, we  consider the following initial-boundary value problem of the
Navier-Stokes system  in $\R_+ \times (0,T)$:
\begin{align}\label{maineq2}
\begin{array}{l}\vspace{2mm}
u_t - \De u + \na p =-\mbox{div}\,(u\otimes u), \qquad \mbox{div} \, u =0, \mbox{ in }
 \R_+\times (0,T),\\
\hspace{25mm}u|_{t=0}= h, \qquad  u|_{x_n =0} = g,
\end{array}
\end{align}
where
 $u=(u_1,\cdots, u_n)$ and  $p$ are unknown velocity and the pressure, respectively, and
   $     g=(g_1,\cdots, g_n), \, h=(h_1,\cdots, h_n)$ are the  given data.

In this paper, we show the unique solvability of the Navier-Stokes system \eqref{maineq2} with initial and boundary  data in anisotropic Besov spaces.
The following  states the main result of this paper.
\begin{theo}
\label{thm3}
 For $0<\al<1$ and $T>0$, let $h\in C^{\alpha}(\R_+),
    \ g\in C^{\alpha,\frac{\al}2   }(\Rn\times (0,T)).$
   We assume that
\begin{equation}
 \label{hp2}
 g|_{t=0}= h|_{x_n=0},\  \mbox{div }h=0,\quad R'g_n\in L^\infty(\Rn\times (0,T)),\quad R'{h}_n \in L^\infty(  \R_+  )
\end{equation}
where,  $R'=(R_1',\cdots,R_{n-1}')$ is the $n-1$ dimensional Riesz operator.
 We also assume that there is  $\tilde{h}\in C^\al(\R)$ an extension of $h$ to $\R$  satisfying that $\mbox{div}\, \tilde{h}=0$, $
  R'\tilde{h}_n \in L^\infty(  \R  ).$
 Then, there is  $T^*$ ($0<T*<T$) such that    the Navier-Stokes system
\eqref{maineq2}  has a  weak solution   $u\in C^{\al,\frac{\al}2 }({\mathbb R}^n_+\times (0,T^*))$  with appropriate distribution $p$.
Moreover, $u$ is a unique in the class $C^{\al,\frac{\al}2 }({\mathbb R}^n_+\times (0,T^*))$.
\end{theo}

There are abundant literature for the solvability of the Navier-Stokes system \eqref{maineq2}  when $g=0$.
 When  $h\in C^{s}(\R_+)$ for $s >2$,  V.A. Solonnikov \cite{sol1}  showed the local in time existence of the unique  solution
 $ u\in C^{s, \frac{s}2 }(\R_+\times (0,T))
 $. See also \cite{maremonti2}. In \cite{sol3}, he also showed the local in time existence of the unique  solution $u\in C(\R_+\times (0,T))$   when $h\in C(\R_+)$.
In \cite{maremonti}, P. Maremonti showed the unique existence of classical solution
of the Navier-tokes system when the initial data is nonconvergent at infinity.
Theorem \ref{thm3} generalizes  the solvability result in \cite{sol1} to a nonzero boundary data $g\in C^{\al,\frac{\al}{2}}(\Rn\times (0,\infty))$  for $0 < \al <1$.

Navier-Stokes system in the half space has been studied mostly in $p$-frame work
 (Here $p$-frame work means function spaces such as $L^p$'s, $W^{k,p}$'s, $1<p<\infty$, or their
  interpolation spaces, and $\infty$-frame work means such as $L^\infty$'s,  $W^k_\infty$'s, or their interpolation spaces).
See \cite{amann,cannone,kozono1,sol1} and references therein for  the solvability of the Navier-Stokes system in the half space with   homogeneous boundary data, that is, with $g=0$.
See \cite{fernandes,amann1,amann2,lewis,voss}  and references therein for the solvability of the Navier-Stokes system  in the half space with  the nonhomogeneous boundary data, that is,  with $g\neq 0$.

There are  many  literatures  for the study of the
 Navier-Stokes system in other domain such as whole space, a bounded domain, or exterior domain (with homogeneous or nonhomogeneous boundary data).
If we mention papers using $\infty$-frame work,  see \cite{giga,koch,kozono, sawada,amann,sol5} and the references therein.
If we mention papers using $p$-frame work, see \cite{amann, amann1, amann2, farwig1, farwig2, farwig3, giga1,grubb, kato} and the references therein.

Although the unsteady Navier-Stokes equations with  low regular boundary data have been studied in several papers such as \cite{fernandes,amann1,amann2,farwig1,farwig2,farwig3}, etc,
 we are interested in finding optimal regularity (in space-time) of the solution corresponding to the given  data.
As a first step we consider the H$\ddot{\rm o}$lder continuous Diriclet boundary data.
In our forthcoming paper we would like to consider optimal regularity (in space-time) of the solution when low regular boundary data is given.

For the proof of Theorem \ref{thm3}, it is necessary to study the initial-boundary value problem of the Stokes system in $\R_+ \times (0,T)$:
\begin{align}\label{maineq}
\begin{array}{l}\vspace{2mm}
u_t - \De u + \na p =f, \qquad div \, u =0, \mbox{ in }
 \R_+\times (0,T),\\
\hspace{20mm}u|_{t=0}= h, \qquad  u|_{x_n =0} = g.
\end{array}
\end{align}

The following states our result on  the unique solvability of the Stokes system \eqref{maineq}.
 \begin{theo}
 \label{thm1}
For $0<\al<1$ and $T>0$, let $h\in C^{\alpha}(\R_+), \,
  \ g\in C^{\alpha,\frac{\al}2 }(\Rn\times (0,T))$  satisfy the same hypotheses as in  Theorem \ref{thm3}.
   Moreover, we assume that $
   f=\mbox{div}\,{\mathcal F},$ where $ {\mathcal F} \in C^{\alpha,\frac{\al}2 }(\R_+\times (0,T))$ with an extension
   %
 $\tilde F\in C^{\alpha,\frac{\al}2 }(\R\times (0,T))$.
 Then, Stokes system
\eqref{maineq} has  a  unique solution  $u\in C^{\al,\frac{\al}2  }({\mathbb R}^n_+\times (0,T))$   with appropriate distribution $p$
 with
  \begin{align}
  \label{es2}
\notag\| u\|_{{C}^{\al,\frac{\al}2 }_{\infty}({\mathbb R}^n_+\times (0,T))}
 \leq& c\Big( \|h\|_{{C}^{\alpha}(\R_+)}+\max\{T^\frac12, T^{\frac{1}{2}+\frac{\al}{2}}\} \| {\mathcal F}\|_{{C}^{\alpha,\frac{\al}2 }_\infty(\R_+\times (0,T))}\\
 &\qquad+\|g\|_{{C}^{\alpha,\frac{\al}2  }(\Rn\times (0,T))} +  \|R'{h}_n  \|_{   L^\infty(  \R )} +\|R'g_n\|_{L^\infty(\Rn\times (0,T))}\Big).
 \end{align}

 \end{theo}

When $h\in {C}^{s} ({\mathbb R}^3_+), f\in  {C}^{s-2,\frac{s}2 -1}({\mathbb R}^3_+\times (0,T))$ and $ g\in C^{s,\frac{s}2 }({\mathbb R}^2 \times (0,T))$  for $s>2$, V.A. Solonnikov \cite{sol2}  showed that there is a unique solution of the Stokes system \eqref{maineq} so that
\[
  \| u\|_{{\dot C}^{s,\frac{s}2 }({\mathbb R}^3_+\times (0,T))}
 \leq c\Big( \|h\|_{\dot{C}^{s}({\mathbb R}^3_+)} +\|f\|_{{\dot C}^{s-2,\frac{s}2-1}({\mathbb R}^3_+\times (0,T))}+\|g\|_{\dot {C}^{s,\frac{s}2}({\mathbb R}^2 \times (0,T))}
  + \| R'(D_t g_{3}) \|_{L^\infty({\mathbb R}^2 ; \dot {C}^{\frac{s}2} (0,T)) } \Big).
\]
Theorem \ref{thm1} generalizes the result of  \cite{sol2} to $0<s <1$.
Our result could be compared with the result in \cite{raymond1}, where  $V^{s,\frac{s}{2}}(\Omega \times (0,T))$ , $0\leq s\leq 2$, has been considered as a solution spaces  in  a bounded domain (see \cite{raymond1} for the definition of $V^{s,\frac{s}{2}}(\Omega \times (0,T))$).

There are various literatures for the solvability of the  Stokes system \eqref{maineq}
with homogeneous boundary data, that is, with $g=0$. See \cite{shimizu,maremonti,maremonti3,sol1,sol3}, and references therein.
In particular, 
 the following estimate is derived in  \cite{sol3}:
\begin{align}
\label{known2}
\| u\|_{L^\infty({\mathbb R}^n_+ \times  (0,T))} \leq c\Big(
\|h\|_{L^\infty({\mathbb R}^{n}_+)}+T^{\frac{1}{2}}\|{\mathcal F}\|_{L^\infty(\R_+\times (0,T))} \Big),
 \end{align}
 where $h\in C(\R_+)$ and $f=\mbox{div}\, {\mathcal F}, {\mathcal F}=(F_{kj})_{j,k=1}^n\in C(\R_+\times (0,T))$ with $ \mbox{div }h=0,  \ h|_{x_n=0}=0, \ (F_{n1},\cdots, F_{nn})|_{x_n=0}=0$. See also \cite{shimizu}.
%

When $f=0$ and  $h=0$,    T.K. Chang and  H.J. Choe \cite{CC}
 showed that
\begin{align}
\label{known1}
\| u\|_{L^\infty({\mathbb R}^n_+ \times  (0,T))} \leq
c \Big( \|g\|_{L^\infty({\mathbb R}^{n-1} \times (0,T))} + \| R' g_n \|_{L^\infty(\Rn \times (0,T)) } \Big),
 \end{align}
where $g \in L^\infty(\Rn\times (0,T)), \ R'g_n\in L^\infty(\Rn\times (0,T)),\ g|_{t=0}=0$. See also \cite{So2,So3}.

We organized  this paper as follows.
In section \ref{notation}, we introduce the  notations and the function spaces such as anisotropic Besov spaces and the anisotropic H$\ddot{\rm o}$lder spaces.
In section \ref{zero}, we  consider Stokes system \eqref{maineq} with  the homogeneous external force and homogeneous initial velocity, and give the proof of Theorem \ref{Rn-1}.
In section \ref{general}, we  complete the proof of Theorem \ref{thm1} with the help of Theorem \ref{Rn-1}.
In section \ref{proof.thm3}, we give the proof of Theorem \ref{thm3} by constructing approximate solutions.

%
%
%

\section{\bf Notations and  Definitions}\label{notation}
\setcounter{equation}{0}

The points of spaces $\Rn$ and $\R$ are denoted by  $x'$ and $x=(x',x_n)$, respectively.
The multiple derivatives are denoted by $ D^{k}_x D^{m}_t = \frac{\pa^{|k|}}{\pa x^{k}} \frac{\pa^{m} }{\pa t}$ for multi index
$ k$ and nonnegative integer $ m$.
For vector field $f=(f_1,\cdots, f_n)$ on $\R$, set $f'=(f_1,\cdots, f_{n-1})$ and $f=(f',f_n)$.
Throughout this paper we denote by $c$ various generic constants.

For the Banach space $X$,  $X'$  denotes the dual space of $X$. 
 {For  the a $m$-dimensional smooth domain $\Omega$,
 $C^\infty_0(\Omega)$ stands for the collection of all complex-valued infinitely differentiable  functions in ${\mathbb R}^m$ compactly supported in $\Omega$. 
%
 Let $1\leq p\leq \infty$ and $k$ be a nonnegative integer.
  The usual Sobolev spaces and homogeneous Sobolev spaces are denoted by  $W^k_p(\Omega)$ and $\dot{W}^k_p(\Omega)$, respectively. Note that $W^0_p(\Omega)=\dot{W}^0_p(\Omega)=L^p(\Omega)$. Let  $0<\al<1$. The usual H$\ddot{\rm o}$lder spaces and the homogeneous H$\ddot{\rm o}$lder spaces are denoted by  $C^{k+\al}(\Omega)$ and $\dot{C}^{k+\al}(\Omega)$, respectively.

 It is known that $C^{k+\al}(\Omega)=B^{k+\al}_\infty(\Omega)$ and $\dot{C}^{k+\al}(\Omega)=\dot{B}^{k+\al}_\infty(\Omega)$ with equivalent norms, where $B^s_p(\Omega)$ and $\dot{B}^s_p(\Omega)$ are the usual Besov spaces and the homogeneous Besov spaces, respectively.
 See \cite{BL,St,Tr,Triebel, Triebel1, Triebel2} for the definition of  Besov spaces and their  properties.

Let $\Om$ be a domain in $m$-dimensional domain and $I$ be an open interval.
Anisotropic H$\ddot{\rm o}$lder spaces
$C^{k+\al,\frac{k+\al}{2}}(\Omega\times I)$  and  homogeneous anisotropic H$\ddot{\rm o}$lder spaces $\dot{C}^{k+\al,\frac{k+\al}{2}}(\Omega\times I)$
are the set of functions on $\Omega \times I$ normed with
\[
\| f\|_{C^{k+\al,\frac{k+\al}2} (\Omega \times I
)} : =\sum_{|l|+2l_0\leq k}\|D^{l_0}_tD^l_xf\|_{L^\infty(\Omega\times I )} + \sum_{|l|+2l_0= k}[D^{l_0}_tD^l_x f]_{\al,\Omega\times I} < \infty
\]
\[
\| f\|_{\dot{C}^{k+\al,\frac{k+\al}2} (\Omega \times I
)} : = \sum_{|l|+2l_0= k}[D^{l_0}_tD^l_x f]_{\al,\Omega\times I} < \infty,
\]
where
\[
[ f]_{\al,\Omega\times I} : = \sup_{t\in I}\sup_{x\neq y \in \Omega}\frac{|f(x,t) -  f(y,t)|}{|x-y|^{\al }}+ \sup_{x\in \Omega}\sup_{s\neq t \in I}\frac{|f(x,t) -  f(x,s)|}{|t-s|^{\frac{\al}{2} }}.
\]


The properties of anisotropic H$\ddot{\rm o}$lder spaces are the same as  the properties of  H$\ddot{\rm o}$lder spaces. For example, ${C}^{k+\al,\frac{k+\al}{2}}(\Omega\times I)=B^{k+\al,\frac{k+\al}{2}}_\infty(\Omega\times I)$ and $\dot{C}^{k+\al,\frac{k+\al}{2}}(\Omega\times I)=\dot{B}^{k+\al,\frac{k+\al}{2}}_\infty(\Omega\times I)$ with equivalent norms, where $B^{s,\frac{s}{2}}_p(\Omega\times I)$ and $\dot{B}^{s,\frac{s}{2}}_p(\Omega\times I)$ are the anisotropic Besov spaces and the homogeneous anisotropic Besov spaces, respectively (For the definition of anisotropic Besov spaces and the homogeneous anisotropic Besov spaces, see \cite{BL, Triebel}).
  The properties of the  anisotropic Besov spaces in $\Omega \times I$ are comparable with the properties of Besov spaces in $\Omega$, whose proof  can be shown by the same arguments as in  \cite{BL,St,Tr,Triebel,Triebel1,Triebel2}.


\begin{defin}[Weak solution to the Stokes system]
Suppose that $f=\mbox{div}{\mathcal F}, \ {\mathcal F}=\{F_{ij}\}_{i,j=1}^n\in C^{\al,\frac{\al}{2}}(\R_+\times (0,T))$, $g\in C^{\al,\frac{\al}{2}}(\Rn\times (0,T))$ and $h\in C^\al(\R_+)$.
Then a vector field $u\in C^{\al,\frac{\al}{2}}(\R_+\times (0,T))$ is called a weak solution of the Stokes system \eqref{maineq} if the following conditions are satisfied:\\
1)
\[
\int^T_0\int_{\R_+}\nabla u:\nabla \Phi dxdt=\int^T_0\int_{\R_+}u\cdot \Phi_t-{\mathcal F}:\nabla \Phi dxdt\]
for each $\Phi\in C^\infty_0(\R_+\times (0,T))$ with $\mbox{div}_x\Phi=0$,
\\
2)
$u(x,0)=h(x) $ in $\R_+$ in trace sense.
\\
3)$u(x',0,t)=g(x',t)$ in $\Rn\times (0,T)$ in trace sense.

\end{defin}

\begin{defin}[Weak solution to the Navier-Stokes system]
Suppose that  $g\in C^{\al,\frac{\al}{2}}(\Rn\times (0,T))$ and $h\in C^\al(\R_+)$.
Then a vector field $u\in C^{\al,\frac{\al}{2}}(\R_+\times (0,T))$ is called a weak solution of the Navier-Stokes system \eqref{maineq2} if the following conditions are satisfied:\\
1)
$\nabla u\in L^\infty(K  \times (\delta, T))$  for each $\delta > 0$ and for each compact subset $K$ of $\R_+$,
\\
2)
\[
\int^T_0\int_{\R_+}\nabla u:\nabla \Phi dxdt=\int^T_0\int_{\R_+}u\cdot (\Phi_t-(\Phi\cdot \nabla)u) dxdt\]
for each $\Phi\in C^\infty_0(\R_+\times (0,T))$ with $\mbox{div}_x\Phi=0$,
\\
3)
$u(x,0)=h(x) $ in $\R_+$ in trace sense.
\\
4)$u(x',0,t)=g(x',t)$ in $\Rn\times (0,T)$ in trace sense.

\end{defin}
\section{\bf Stokes system with homogeneous external force and  initial velocity}

\label{zero}

\setcounter{equation}{0}

Let us   consider the following initial-boundary value problem of a
nonstationary Stokes system  in $\R_+ \times (0,T)$:
\begin{align}\label{maineq1}
\begin{array}{l}\vspace{2mm}
w_t - \De w + \na q =0, \qquad div \, w =0, \mbox{ in }
 \R_+\times (0,T),\\
\hspace{20mm}w|_{t=0}= 0, \qquad  w|_{x_n =0} = G.
\end{array}
\end{align}

In \cite{sol1}, an explicit formula for $w$
 of the Stokes system \eqref{maineq1} with boundary data $G=(G', 0)$  is
 obtained by
\begin{align}\label{simple}
w_i(x,t)& = \sum_{j=1}^{n-1}\int_0^t \int_{\Rn} K_{ij}( x'-y',x_n,t-s)G_j(y',s) dy'ds, \\
\label{simple1}
q(x,t)& = \sum_{j=1}^{n-1}\int_0^t \int_{\Rn} \pi_j(x'-y',x_n,t-s) G_j(y',s) dy'ds
.
\end{align}
Here,
\begin{align*}
K_{ij}(x,t) & = -2 \delta_{ij}D_{x_n}  \Ga(x,t)  +4 D_{x_j}\int_0^{x_n} \int_{\Rn}  D_{z_n}  \Ga(z,t)  D_{x_i} N(x-z)  dz,\\
\pi_j (x,t) & =-2\delta(t)D_{x_j} D_{x_n}N(x)
      +4D_{x_j}D^2_{ x_n}A(x,t)
+4 D_t D_{x_j} A(x,t),\\
 A(x,t) & =\int_{\Rn}\Ga(z',0,t)N(x'-z',x_n)dz',
\end{align*}
where $  \Ga$ and $N$ are fundamental solutions of heat equation and Laplace equation in $\R$, respectively, that is,
\[
 \Gamma(x,t)=\left\{\begin{array}{ll} \vspace{2mm}
 \frac{c}{ (2\pi t)^{\frac{n}{2}}}e^{-\frac{|x|^2}{4t}}&\mbox{ if }t>0,\\
 0& \mbox{ if }t\leq 0,
 \end{array}\right. \quad \mbox{and} \quad    N(x) = \left\{\begin{array}{ll}
 \vspace{2mm}
  \frac{1}{\om_n (2-n)|x|^{n-2}}&\mbox{ if }n\geq 3,\\
 \frac{1}{2\pi}\ln |x|&\mbox{ if }n=2.\end{array}\right.
 \]

 \begin{theo}\label{Rn-1}
Let  $0<\alpha<1 $.
Let  $G \in \dot{ C}^{\al,\frac{\al}2 }({\mathbb R}^{n-1} \times (0,\infty))$ with
$G_n=0$.  We also assume that $G\Big|_{t=0} =0$.
Then, the function $w$ defined by \eqref{simple} is in $\dot{C}^{\al,\frac{\al}2}({\mathbb R}^{n-1}_+\times (0,\infty))$  and satisfies
\[
\| w\|_{\dot{C}^{\al,\frac{\al}2 }({\mathbb R}^n_+\times (0,\infty))}
 \leq c\|G\|_{\dot{C}^{\al,\frac{\al}2  }({\mathbb
R}^{n-1} \times (0,\infty))}.
 \]
\end{theo}


\begin{proof}

According to the result of  V.A. Solonnikov \cite{sol2}, if  $ G=(G',0)\in \dot{C}^{s,\frac{s}2 }({\mathbb R}^2 \times (0,T
))$  for $s>2$, then $w$ defined by \eqref{simple} satisfies
\[
  \| w\|_{{\dot C}^{s,\frac{s}2 }({\mathbb R}^3_+\times (0,T))}
 \leq c\|G\|_{\dot {C}^{s,\frac{s}2}({\mathbb R}^2 \times (0,T))}, \ s>2.
\]
%
The argument in \cite{sol2} can be applied for any $n\geq 2$ to obtain the same estimates as the above.

According to the result of  T.K. Chang and  H.J. Choe \cite{CC}, if $G=(G',0)\in L^\infty(\Rn \times (0,T)),$ then
 $w$ defined by \eqref{simple} satisfies
\[
\| w\|_{L^\infty({\mathbb R}^n_+ \times  (0,T))} \leq
c  \|G\|_{L^\infty({\mathbb R}^{n-1} \times (0,T))}.
 \]
Interpolate the  above two estimate, then we obtain the estimate in Theorem \ref{Rn-1}.

\end{proof}

\begin{rem}\label{rem1031}
Let $G_{ij}= D_{x_j}\int_0^{x_n} \int_{\Rn}  D_{z_n}  \Ga(z,t)  D_{x_i} N(x-z)  dz$.
It is known that \begin{align}
|D^{l}_{x_n} D^{k0}_{x'} D_{t}^{m} G_{ij}(x,t)|
& \leq  \frac{c}{t^{m + \frac12} (|x|^2 +t )^{\frac12 n + \frac12 k} (x_n^2 +t)^{\frac12 l}},
\end{align}
where $ 1 \leq  i \leq n$ and $1 \leq j \leq n-1$ (see \cite{sol1}).
Using the properties of Heat kernel $\Gamma$ and the estimates of $G_{ij}$,
 it is easy to see that
\begin{align*}
x_n^{-k +1-\al  }t^{-\frac{1}{2}}|D_{x}^{k} w(x,t)| \leq c\|G\|_{\dot{C}^{\al,\frac{\al}2  }({\mathbb
R}^{n-1} \times (0,T))}.
\end{align*}
Therefore, $w$ is smooth in $\R_+$ for each $t>0$.

\end{rem}

\section{\bf Stokes system with nonhomogeneous  external force and initial velocity}
\label{general}
\setcounter{equation}{0}

In this section we   consider the Stokes system \eqref{maineq} with  $f=\mbox{div}{\mathcal F}, g, h$ satisfying the hypothesis of Theorem \ref{thm1}.
\subsection{\bf Formal decompositions}

\label{stokesformula}

Let
$\tilde{ \mathcal  F}$ be an extension of ${\mathcal F}$ to $\R\times (0,T)$
 and let $\tilde{f}=\mbox{div }\tilde{\mathcal F}$.
Let ${\mathbb P}$ be  the Helmholtz projection operator  on $\R$ defined by
\[
[{\mathbb P}\tilde{f}]_j(x,t)=\delta_{ij} \tilde f_i+\int_{\R} D_{x_i}D_{x_j}N(x-y)\tilde{f}_i(y,t)dy=\delta_{ij}\tilde{f}_i+R_iR_j\tilde{f}_i
\]
and define ${\mathbb Q}$ by
\[
{\mathbb Q}\tilde{f}=-\int_{\R} D_{x_i}N(x-y)\tilde{f}_i(y,t)dy.
\]
Then, we have
 \[
    \mbox{div}\,{\mathbb P}\tilde{f}=0\mbox{ in }\R\times (0,T) \quad \mbox{      and      } \quad
  \tilde{f}={\mathbb P}\tilde{f} +D_x{\mathbb Q}\tilde{f}.\]
 Note that $[{\mathbb P}\tilde{f}]_j=D_{x_k}[\delta_{ij}\tilde{F}_{ki}+R_iR_j\tilde{F}_{ki}]$ for  $\tilde{f}=\mbox{div}\tilde{\mathcal F}$.
Define  $V$ by
\begin{align}\label{V}
  V_j(x,t)=\int^t_{0}\int_{\R}D_{x_k}\Gamma(x-y,t-s)[\delta_{ij}\tilde{F}_{ki}+R_iR_j\tilde{F}_{ki}](y,s)dyds.
 \end{align}
   Observe that
  $V$  satisfies the equations
 \[
 \begin{array}{l}\vspace{2mm}
V_t - \De V  ={\mathbb P}\tilde{f},\ \mbox{div} \, V=0 \mbox{ in }
 \R\times (0,T), \\
\hspace{20mm}V|_{t=0}= 0\mbox{ on }\R.
\end{array}
\]

Let
 $\widetilde{h}$ be an  extension of  $h$ satisfying that $
 \mbox{div} \, \widetilde{h}=0\mbox{ in }\R$.
Define $v$ by
   \begin{equation}\label{small-v}
  v(x,t)=\int_{\R}\Gamma(x-y,t) \widetilde{h}(y)dy.
  \end{equation}
Observe that $v$ satisfies the equations
 \[
\begin{array}{l}\vspace{2mm}
v_t - \De v  =0,\ \mbox{div} \, v=0 \mbox{ in }
 \R\times (0,T),\\
\hspace{20mm}v|_{t=0}= \tilde{h} \mbox{ on }\R.
\end{array}
\]
Define $\phi$ by
 \begin{equation}\label{phi}
 \phi(x,t)=2\int_{\Rn}N(x'-y',x_n)\Big(g_n(y',t)-v_n(y',0,t)-V_n(y',0,t)\Big)dy'.
 \end{equation}
  Observe that
 \begin{equation*}
  \Delta \phi=0, \qquad  \nabla\phi|_{x_n=0}=(R'(g_n-v_n|_{x_n=0}-V_n|_{x_n=0}), \,  g_n-v_n|_{x_n=0}-V_n|_{x_n=0}).
 \end{equation*}
Moreover, note that $ \nabla\phi|_{t=0}=0$  if $g_n|_{t=0}=h_n|_{x_n=0}$. Let
\begin{equation}
\label{G}
G =(G',0),\mbox{ where }G'=g' -V'|_{x_n=0} -v'|_{x_n=0} - R'(g_n -v_n|_{x_n=0}-V_n|_{x_n=0}).\end{equation}
Note that
   $G'|_{t=0}=0$ if $g|_{t=0}=h|_{x_n=0}$.

 Finally, let $(w,q)$ be defined by \eqref{simple} and \eqref{simple1} with $G$ defined by \eqref{G}. 
 Then, $
 u=w+\nabla\phi+v+V\mbox{ and } p=q-\phi_t+{\mathbb Q}\tilde{f}
$  satisfies formally  the  nonstationary  Stokes system  \eqref{maineq}.

\subsection{Preliminary Estimates}

The subsequent propositions are the basic tool for the estimate of $v, V, \nabla \phi$ and $w$ introduced in the previous section.
See Appendix \ref{appen2}, Appendix \ref{appen3}, Appendix \ref{appen-riesz} and Appendix \ref{appen-poisson} for the proof of the Proposition \ref{proheat1}, Proposition \ref{propheat2}, Proposition \ref{propriesz} and Proposition \ref{proppoisson2},  respectively.
\begin{prop}
\label{proheat1} Let $0<\al$
.
For $f\in {B}_\infty^{\alpha}(\R)$, define $u(x,t)=\int_{\R}\Gamma(x-y,t)f(y)dy$.
Then $u\in {B}_\infty^{\alpha, \frac{\al}2  }(\R\times (0,T))$ with
\[
\|u\|_{\dot{B}_\infty^{\al,\frac{\al}2  }(\R\times (0,\infty)}\leq c\|f\|_{\dot{B}_\infty^{\alpha}(\R)},\qquad
\|u\|_{L^\infty(\R\times (0,\infty))}\leq c\|f\|_{L^\infty(\R)}.
\]

Moreover, $u$ is smooth in $\R\times (0,T)$ with
\[
\sup_{x\in \R, t\in (0,\infty)}t^{m+\frac{|k|}{2}}|D^m_tD^k_xu(x,t)|\leq c\|f\|_{L^\infty(\R)}.\]
\end{prop}

\begin{prop}
\label{propheat2}
Let $0<\al$.
Let $f\in {B}_\infty^{\alpha,\frac{\al}2  }(\R\times (0,T))$. Define $u(x,t)
=\int^t_0\int_{\R} D_x\Gamma(x-y,t-s)f(y,s)dyds$.
Then $u\in {B}_\infty^{\al+1,\frac{\al}2+\frac{1}{2}  }(\R\times (0,T))$
with
\begin{align*}
\|u\|_{\dot{B}_\infty^{\al+1,\frac{\al}2+\frac{1}{2} }(\R\times (0,T))}&\leq c\|f\|_{\dot{B}_\infty^{\alpha,\frac{\al}2  }(\R\times (0,T))},\\
\|u\|_{\dot{B}_\infty^{\al,\frac{\al}2 }(\R\times (0,T))}&\leq cT^{\frac{1}{2}}\|f\|_{\dot{B}_\infty^{\alpha,\frac{\al}2  }(\R\times (0,T))},\\
\|u\|_{L^\infty(\R\times (0,T))}&\leq cT^{\frac12   }\|f\|_{L^\infty( (0,T);BMO(\R))}.
\end{align*}
Here $BMO(\R)$ denotes the usual BMO space, which is the dual space of Hardy space ${\mathcal H}^1(\R)$.
%
%

\end{prop}

\begin{prop}
\label{propriesz}

Let  $ \al \in {\mathbb R}$.
Then
\[
\|Rf\|_{\dot{B}^{\al}_\infty({\mathbb R}^n)}\leq c\|f\|_{\dot{B}^{\al}_{\infty}({\mathbb R}^n)},\
\|Rf\|_{BMO(\R)}\leq c\|f\|_{BMO(\R)},\ \|Rf\|_{{\mathcal H}^1(\R)}\leq c\|f\|_{{\mathcal H}^(\R)}.
\]

Moreover, if $0<\al$, then
\[
\|Rf\|_{\dot{B}_\infty^{\al,\frac{\al}2  }
({\mathbb R}^n\times (0,T))}\leq c\|f\|_{\dot{B}_\infty^{\al,\frac{\al}2  }({\mathbb R}^n\times (0,T))}.
\]

\end{prop}

%

\begin{prop}\label{proppoisson2}
Let $0<\al$.
Define $Pf(x,t)=\int_{\Rn}\frac{x_n}{(|x'-y'|^2+x_n^2)^{\frac{n}{2}}}f(y',t)dy'.$
Then
\begin{align*}
\|Pf(t)\|_{\dot{B}_\infty^{\al}(\R_+)}&\leq c\|f(t)\|_{\dot{B}_\infty^\al(\Rn)},\\
\|Pfu(t)\|_{L^\infty(\R_+)}&\leq c\|f(t)\|_{L^\infty(\Rn)},\\
\| Pf \|_{\dot{B}_\infty^{\al,\frac{\al}2  }(\R_+\times (0,T))}&\leq c\|f\|_{\dot{B}_\infty^{\al,\frac{\al}2  }(\Rn \times (0,T))}.
\end{align*}


Moreover, $u(t)$ is smooth in $\R$ with
\[
\sup_{x\in \R}x_n^k|D^{k}_{x}u(x,t)|\leq c\|f(t)\|_{L^\infty(\Rn)}.
\]

\end{prop}

%
%
%
%

\subsection{Proof of Theorem \ref{thm1}}
\label{proof.thm3}

Choose  $\tilde{h}\in \dot{C}^{\al}(\R)$ and $\tilde{\mathcal F}\in \dot{C}^{\al,\frac{\al}{2}}(\R\times (0,T))$ which are the extension of $h$ and ${\mathcal F}$, respectively.
Let   $V$, $v$ and $ \phi$ be the corresponding vector fields defined by \eqref{V},  \eqref{small-v}, and \eqref{phi}, respectively, and  
%
let $w$ be  defined by \eqref{simple} with $G$ as   \eqref{G}.

$\bullet$
At this step, we will show that $u=v+V+\nabla \phi+w\in \dot{C}^{\al,\frac{\al}{2}}(\R_+\times (0,T)).$

 From Proposition \ref{proheat1} and the property of the extension $\widetilde{h}$
\begin{equation}\label{0917-v}
\| v\|_{\dot{C}^{\al,\frac{\al}{2}}(\R\times (0,T))}
\leq c\|h\|_{\dot{C}^{\al}(\R_+)}.
\end{equation}
From Proposition \ref{propheat2}, Proposition \ref{propriesz} and the property of the extension $\widetilde{f}$
\begin{equation}\label{0917-V}
 \| V\|_{\dot{C}^{\al,\frac{\al}{2}}(\R\times (0,T))} \leq c T^{\frac{1}{2}}\|\delta_{ij}\tilde{F}_{ik}+R_iR_j\tilde{F}_{ik}\|_{\dot{C}^{\al,\frac{\al}{2}}(\R\times (0,T))}
\leq cT^{\frac{1}{2}}\|  {\mathcal F}\|_{\dot{C}^{\al,\frac{\al}{2}}(\R_+\times (0,T))}.
\end{equation}

According to the well known trace theorem, $V,\ v\in \dot{C}^{\al,\frac{\al}{2}}(\R\times (0,T))$  imply $V|_{x_n=0},\quad  v|_{x_n=0}\in \dot{C}^{\al,\frac{\al}{2}}(\Rn\times (0,T))$ with
\begin{equation}\label{0917-1}
\begin{array}{ll}\|V|_{x_n=0}\|_{\dot{C}^{\al,\frac{\al}{2}}(\Rn\times (0,T))}&\leq c\| V\|_{\dot{C}^{\al,\frac{\al}{2}}(\R_+\times (0,T))},\\
  \|v|_{x_n=0}\|_{\dot{C}^{\al,\frac{\al}{2}}(\Rn\times (0,T))}&\leq c\| v\|_{\dot{C}^{\al,\frac{\al}{2}}(\R_+\times (0,T))}.\end{array}
\end{equation}
Again, according to Proposition \ref{propriesz}, $g_n,\ V|_{x_n=0},\ v|_{x_n=0}\in \dot{C}^{\al,\frac{\al}2  }(\Rn\times (0,T))$ imply
$R'g_n,$ $ R'(v_n|_{x_n=0})$, $R'(V_n|_{x_n=0})\in \dot{C}^{\alpha,\frac{\al}2  }(\Rn\times (0,T))$ with
\begin{align}\label{0917-2}
\begin{array}{ll}\|R'g_n\|_{ \dot{C}^{\al,\frac{\al}{2}}(\Rn\times (0,T))  }&\leq c\|g_n\|_{ \dot{C}^{\al,\frac{\al}{2}}(\Rn\times (0,T))  },
\\
 \|R'(v_n|_{x_n=0})\|_{ \dot{C}^{\al,\frac{\al}{2}}(\Rn\times (0,T)) }&\leq c\|v_n|_{x_n=0}\|_{ \dot{C}^{\al,\frac{\al}{2}}(\Rn\times (0,T)) },\\
  \|R'(V_n|_{x_n=0})\|_{ \dot{C}^{\al,\frac{\al}{2}}(\Rn\times (0,T))}&\leq c\|V_n|_{x_n=0}\|_{ \dot{C}^{\al,\frac{\al}{2}}(\Rn\times (0,T))}.\end{array}
\end{align}

Observe that
\[
D_{x_n}\phi=2\int_{\Rn}D_{x_n}N(x'-y',x_n)\Big(g_n(y',t)-v_n(y',0,t)-V_n(y',0,t)\Big)dy',\]
\[
D_{x'}\phi=2\int_{\Rn}D_{x_n}N(x'-y',x_n)\Big(R'g_n(y',t)-R'v_n(y',0,t)-R'V_n(y',0,t)\Big)dy'.
\]
According to the Proposition \ref{propriesz} and  Proposition \ref{proppoisson2}, \eqref{0917-1} and \eqref{0917-2} imply $\nabla \phi\in \dot{C}^{\al,\frac{\al}2  }(\R_+\times(0,T))$ with
\begin{equation}\label{0917-phi}
\|\nabla \phi\|_{\dot{C}^{\al,\frac{\al}{2}}(\R_+\times(0,T))}
\leq c \Big(\|g_n\|_{\dot{C}^{\al,\frac{\al}{2}}(\Rn\times(0,T))}
+\|v_n\|_{\dot{C}^{\al,\frac{\al}{2}}(\R\times(0,T))}
 +\|V_n\|_{\dot{C}^{\al,\frac{\al}{2}}(\R\times(0,T))} \Big).
\end{equation}

\eqref{0917-1} and \eqref{0917-2} also imply $G\in \dot{C}^{\al,\frac{\al}{2}}(\R_+\times (0,T))$ with
\begin{equation}
\label{0917-2222}
\|G\|_{\dot{C}^{\al,\frac{\al}{2}}(\Rn\times (0,T))}
 \leq c\Big(\|v\|_{\dot{C}^{\al,\frac{\al}{2}}(\R\times (0,T))}+\|V\|_{\dot{C}^{\al,\frac{\al}{2}}(\R\times (0,T))}+\|g\|_{\dot{C}^{\al,\frac{\al}{2}}(\Rn\times (0,T))} \Big).
\end{equation}
Applying  Theorem \ref{Rn-1},   we have
\begin{equation}\label{0917-w}
\| w\|_{\dot{C}^{\al,\frac{\al}{2}}({\mathbb R}^n_+\times (0,T))}
\leq c\|G\|_{\dot{C}^{\al,\frac{\al}{2}}({\mathbb
R}^{n-1} \times (0,T))}
.
 \end{equation}

Combining  \eqref{0917-v}, \eqref{0917-V}, \eqref{0917-phi} and \eqref{0917-w} together with \eqref{0917-2222}, we conclude that $u=v+V+\nabla \phi+w\in \dot{C}^{\al,\frac{\al}{2}}(\R_+\times (0,T))$   with the inequality 
  \begin{align}
  \label{es2-2}
\| u\|_{\dot{C}^{\al,\frac{\al}2 }_{\infty}({\mathbb R}^n_+\times (0,T))}
 \leq c\Big( \|h\|_{\dot{C}^{\alpha}(\R_+)}+T^\frac12 \| {\mathcal F}\|_{\dot{C}^{\alpha,\frac{\al}2 }_\infty(\R_+\times (0,T))}
+\|g\|_{\dot{C}^{\alpha,\frac{\al}2  }(\Rn\times (0,T))}\Big).
 \end{align}

$\bullet$ At this step, we will show that $u=v+V+\nabla \phi+w\in L^\infty(\R_+\times (0,T))$.

By Proposition \ref{propheat2}, we have
\begin{align}\label{0917-V-infty}
\|  V\|_{L^\infty(\R\times (0,T))} & \leq c  T^{\frac{1}{2}}\|\delta_{ij}\tilde{F}_{ik}+R_iR_j\tilde{F}_{ik}\|_{L^\infty(0,T;BMO(\R))}  \leq c T^{\frac{1}{2}}\|{\mathcal F}\|_{L^\infty(\R_+\times (0,T))},
\end{align}
and
by Proposition \ref{proheat1} we have
\begin{align}\label{0917-v-infty}
\| v\|_{L^\infty(\R\times (0,T))}
\leq c\|h\|_{L^\infty(\R_+)}.
\end{align}

To show  that $\nabla\phi,\  w$ are in $ L^\infty(\R_+\times (0,T))$, it is necessary that  $v_n|_{x_n=0},$ $ V_n|_{x_n=0}, $ $R'(v_n|_{x_n=0}), $ $R'(V_n|_{x_n=0})$ are  in $L^\infty(\R_+\times (0,T))$.
%
Observe that from Proposition \ref{proppoisson2},
\begin{align*}
V(x',0,t)&=\int^t_0\int_{\R}D_{x_k} \Gamma(x'-y',y_n,t-s) [\delta_{ij}\tilde{F}_{ik}+R_iR_j\tilde{F}_{ik}](y,s)dyds\\
&\leq c\int^t_0\|D_{x_k} \Gamma(x'-\cdot,\cdot,t-s)\|_{{\mathcal H}^1(\R)}\|\delta_{ij}\tilde{F}_{ik}+R_iR_j\tilde{F}_{ik}](\cdot,s)\|_{BMO(\R)}ds,
\\
v(x',0,t)&=\int_{\R}\Gamma(x'-y',y_n,t)\tilde{h}(y)dy\leq c\| \Gamma(x'-\cdot,\cdot,t)\|_{L^1(\R)}\|\tilde{h}\|_{L^\infty(\R)}.
\end{align*}
Hence $V|_{x_n=0},\quad  v|_{x_n=0}\in L^\infty(\Rn\times (0,T))$ with
\begin{equation}\label{0917-11}
\|V|_{x_n=0}\|_{L^\infty(\Rn\times (0,T))}\leq cT^{\frac{1}{2}}\| \mathcal F\|_{L^\infty(\R_+\times (0,T))},\quad  \|v|_{x_n=0}\|_{L^\infty(\Rn\times (0,T))}\leq c\|h\|_{L^\infty(\R_+)}.
\end{equation}

The estimates of $R'(v_n|_{x_n=0}), R'(V_n|_{x_n=0})$ in $L^\infty(\R_+\times (0,T))$ are rather delicate.
Note that ${\mathbb P}\tilde{f}_n=\Delta N*{\mathbb P}\tilde{f}_n=\sum_{l\neq n} D_{x_l}^2 N*{\mathbb P}\tilde{f}_n-\sum_{l\neq n}D_{x_n}D_{x_l}{\mathbb P}\tilde{f}_l
$, since $\mbox{div }{\mathbb P}\tilde{f}=0$.
Hence, we have
\begin{align*}
V_n(x',0,t)
&=\sum_{l\neq n} \int^t_0 \int_{\R}D_{x_l}\Gamma(x'-y',y_n,t-s)D_{y_l}D_{y_k}[N*(\delta_{in}+R_iR_n)\tilde{F}_{ki}](y,s)dyds\\
&\quad +\sum_{l\neq n}\int^t_0 \int_{\R}D_{y_l}\Gamma(x'-y',y_n,t-s)D_{y_n}D_{y_k}[N*(\delta_{il}+R_iR_l)\tilde{F}_{ki}]dyds
\\
&=\sum_{l\neq n} \int^t_0 \int_{\R}D_{x_l}\Gamma(x'-y',y_n,t-s)R_lR_k(\delta_{in}+R_iR_n)\tilde{F}_{ki}](y,s)dyds
 \\
 &\quad+\sum_{l\neq n}\int^t_0 \int_{\R}D_{y_l}\Gamma(x'-y',y_n,t-s)R_nR_k(\delta_{il}+R_iR_l)\tilde{F}_{ki}](y,s)dyds.
\end{align*}
Using the above representation, $R'\Big(V_n(t)|_{x_n=0}\Big)$ have the following representation:
\begin{align*}
&R'\Big(V_n(t)|_{x_n=0}\Big) \\
&=\sum_{l\neq n}\int^t_0\int^\infty_0 k(y_n,t-s) \Big(\int_{\Rn}R'D_{x_l}K(x'-y',t-s)R_lR_k(\delta_{in}+R_iR_n)\tilde{F}_{ki})(y',y_n,s)dy'\Big)dy_nds\\
 &\quad+\sum_{l\neq n}\int^t_0\int^\infty_0 k(y_n,t-s) \Big(\int_{\Rn}R'D_{y_l}K(x'-y',t-s)R_nR_k(\delta_{il}+R_iR_l)\tilde{F}_{ki})(y',y_n,s)dy'\Big)dy_nds.
\end{align*}
 Here $K(x',t)=K_t(x')=\frac{1}{(2\pi t)^{\frac{n-1}{2}}}e^{-\frac{|x'|^2}{4t}},\
 k(x_n,t)=k_t(x_n)=\frac{1}{\sqrt{2\pi t}}e^{-\frac{x_n^2}{4t}}.$
Note that
\begin{align*}
&\int_{\Rn}R'D_{x_l}K(x'-y',t-s)R_lR_k(\delta_{in}+R_iR_n)\tilde{F}_{ki})(y',y_n,s)dy'\\
 &\leq c\|D_{x_l}K(x'-\cdot,t-s)\|_{\dot{B}_1^{-\al}(\Rn)}\|R_lR_k(\delta_{ij}+R_iR_j)\tilde{F}_{ki})(\cdot,y_n,s)\|_{\dot{B}_\infty^{\al}(\Rn)}\\ & \leq c\|K(x'-\cdot,t-s)\|_{\dot{B}_1^{1-\al}(\Rn)}\sup_{y_n}\|R_lR_k(\delta_{ij}+R_iR_j)\tilde{F}_{ki})(\cdot,y_n,s)\|_{\dot{B}_\infty^{\al}(\Rn)}\\
 & \leq c(t-s)^{-\frac{1}{2}+\frac{\al}2  }\|R_lR_k(\delta_{ij}+R_iR_j)\tilde{  F}(\cdot,s)\|_{\dot{B}_\infty^{\al}(\R)}\leq c(t-s)^{-\frac{1}{2}+\frac{\al}2  }\|\tilde{  \mathcal F}(\cdot,s)\|_{\dot{B}_\infty^{\al}(\R)}.
\end{align*}
Here we use the fact that  $R: \dot{B}_\infty^{\al}(\R)\hookrightarrow \dot{B}_\infty^{\al}(\R)$ is bounded operator for $\al\in {\mathbb R}$,
$L^\infty({\mathbb R};\dot{B}_\infty^{\al}(\Rn))\cap L^\infty(\Rn;\dot{B}_\infty^{\al}({\mathbb R}))=\dot{B}_\infty^{\al}(\R)\mbox{ for }\al>0,$
 and $\dot{B}_\infty^{\al}(\R)=\dot{C}^\al(\R)$ for $0<\al<1$.
 Hence, we have
\begin{equation} \label{0912-2}
 \|R'\Big(V_n(t)|_{x_n=0}\Big)\|_{L^\infty(\Rn\times (0,T))}
 \leq cT^{\frac12  + \frac{\al}2  }\|\tilde{ \mathcal  F}\|_{\dot{C}^{\al,\frac{\al}2  }(\R\times (0,T))}.
\end{equation}

Direct computation shows the identity
\[
R'\Big(v_n(t)|_{x_n=0}\Big)=\int_{\R}\Gamma(y,t)(R'
\tilde{h}_n)(x'-y',y_n) dy.\]
Hence, if $R'\tilde{h}_n\in L^\infty(\R)$, then
\begin{align}\label{0912-3}
\|R'\Big(v_n(t)|_{x_n=0}\Big)\|_{L^\infty(\Rn\times (0,T))}
\leq c\|R'\tilde{h}_n\|_{L^\infty(\R)}.
\end{align}

By Proposition \ref{proppoisson2}, \eqref{0912-2} and \eqref{0912-3} imply $\nabla \phi\in L^\infty(\R_+\times (0,T))$ with
 \begin{align}
 \label{0917-phi-infty1}
\|D_{x_n}\phi(t)\|_{L^\infty(\R_+)} 
&\leq c  \Big(  \|g_n\|_{L^\infty(\Rn \times (0,T))}+T^\frac12 \|{\mathcal F}\|_{L^\infty(\R_+\times (0,T))}+\|h\|_{L^\infty(\R_+)} \Big),\\
 \|D_{x'}\phi(t)\|_{L^\infty(\R_+)}
 \label{0917-phi-infty2}
 & \leq c \Big( \|R'g_n(t)\|_{L^\infty(\Rn)}+T^{\frac12 +\frac{\al}2  }\|{\mathcal F}\|_{\dot{C}^{\al,\frac{\al}2  }(\R_+\times (0,T))}+\|R'\tilde{h}_n\|_{L^\infty(\R_+)}\Big).
 \end{align}

\eqref{0912-2} and \eqref{0912-3} also imply  $G \in  L^\infty(\Rn\times (0,T))$ with
\begin{align}
\label{0917-w-infty3}
\notag\|G\|_{L^\infty(\Rn \times (0,T))} &\leq c  \Big(  \|g\|_{L^\infty(\Rn \times (0,T))} +\|h\|_{L^\infty(\R_+)}  +T^\frac12 \|{\mathcal F}\|_{L^\infty(\R_+\times (0,T))} \\
& \quad + T^{\frac{1}{2}+\frac{\al}{2}}\| {\mathcal F}\|_{\dot{C}^{\al,\frac{\al}2  }(\R_+\times (0,T))}+ \|R'g_n(t)\|_{L^\infty(\Rn)}+\|R'\tilde{h}_n\|_{L^\infty(\R_+)}\Big)1.
\end{align}
Note that $G_n=0 $, $G|_{t=0}=0$.
According to the result of  \cite{CC},
 \begin{align}
\label{0917-w-infty}
\| w\|_{L^\infty({\mathbb R}^n_+ \times  (0,T))}  \leq
c\|G\|_{L^\infty({\mathbb R}^{n-1} \times (0,T))}
.
 \end{align}

 Combining  \eqref{0917-V-infty}, \eqref{0917-v-infty}, \eqref{0917-phi-infty1},  \eqref{0917-phi-infty2} and \eqref{0917-w-infty} together with \eqref{0917-w-infty3}, we conclude that $u=v+V+\nabla \phi+w\in L^\infty(\R_+\times (0,T))$ with the inequality
  \begin{align}\label{es2-1}
\notag  \| u\|_{L^{\infty}({\mathbb R}^n_+\times (0,T))}
& \leq c \Big(\|h\|_{L^\infty(\R_+)}+T^\frac12 \|{\mathcal F}\|_{L^\infty(\R_+\times (0,T))}
 +{T}^{\frac12+\frac{\al}{2}    }\|  {\mathcal F}\|_{ {C}^{\al,\frac{\al}2 }_\infty(\R_+\times (0,T))}
+\|g\|_{L^\infty(\Rn\times (0,T))}\\
& \quad  +  \|R'\tilde{h}_n  \|_{   L^\infty(  \R )} +\|R'g_n\|_{L^\infty(\Rn\times (0,T))}\Big).
 \end{align}
Combining \eqref{es2-2} and \eqref{es2-1}, we obtain the estimates \eqref{es2} in  Theorem \ref{thm1}.
  The uniqueness follows from \eqref{es2}.

\begin{rem}
Recall that  formally, $u=w+\nabla \phi+v+V$, $p=q-\phi_t+{\mathbb Q}f$ satisfies the Stokes system \eqref{maineq} formally,
%
and $q$ can be written by $q(x,t)=q_0(x,t)+D_t q_1(x,t)$ (see \cite{koch} for the details),  where $q_0(t), q_1(t)$ are harmonic function in $x$ variable, but their  differentiability in $t$ variable is equal to  the differentiability of $g$ in $t$ variable. Therefore, for $g\in C^{\al,\frac{\al}{2}}(\Rn\times (0,T)), \ 0<\al<1$, $q$ is not a function but a distribution in terms of $t$ variables.
From this reason, our solution $u\in C^{\al,\frac{\al}{2}}(\R_+\times (0,T))$  satisfies weak formulation of the Stokes system \eqref{maineq}, but cannot satisfy the  Stokes system \eqref{maineq} in classical sense.

\end{rem}

\section{\bf Proof of Theorem \ref{thm3}}

\label{proof.thm3}
\setcounter{equation}{0}



Since
\[
|(uv)(x,t)-(u v)(y,t)|\leq |v(x,t)||u(x,t)-u(y,t)|+|u(y,t)||v(x,t)-v(y,t)|
\]
and
\[
|(u v)(x,t)-(u v)(x,s)|\leq |v(x,t)||u(x,t)-u(x,s)|+|u(x,s)||v(x,t)-v(x,s)|,
\]
 the following bilinear estimate can be  obtained.
\begin{lemm}
\label{bilinear}

Let $u,v\in C^{\al,\frac{\al}2 }(\R\times[0,T))$.
 Then
\[
\|uv\|_{\dot C^{\al,\frac{\al}2 }(\R\times (0,T))}\leq c\|u\|_{L^\infty(\R\times (0,T))}\|v\|_{\dot{C}^{\al,\frac{\al}2 }(\R\times (0,T))}
+c\|v\|_{L^\infty(\R\times (0,T))}\|u\|_{\dot{C}^{\al,\frac{\al}2 }y(\R\times (0,T))}.\]

\end{lemm}

%
%

\subsection{Approximate solutions}


Let $(u^1,p^1)$ be the solution of the system
\[
\begin{array}{l}\vspace{2mm}
u^1_t - \De u^1 + \na p^1 =0, \qquad div \, u =0, \mbox{ in }
 \R_+\times (0,T),\\
\hspace{20mm}u^1|_{t=0}= h, \qquad  u^1|_{x_n =0} = g.
\end{array}
\]
Let $m\geq 1$.
After obtaining $(u^1,p^1),\cdots, (u^m,p^m)$ construct $(u^{m+1}, p^{m+1})$ which satisfies the system
\[
\begin{array}{l}\vspace{2mm}
u^{m+1}_t - \De u^{m+1} + \na p^{m+1} =f^m, \qquad div \, u^{m+1} =0, \mbox{ in }
 \R_+\times (0,T),\\
\hspace{30mm}u^{m+1}|_{t=0}= h, \qquad  u^{m+1}|_{x_n =0} =g,
\end{array}
\]
where $f^m=-\mbox{div}(u^m\otimes u^m)$.

\subsection{Uniform boundesness}
Let $T\leq 1$.
By the result of Theorem \ref{thm1},   we have
\begin{align*}
\| u^{m+1}\|_{{C}^{\al,\frac{\al}2 }({\mathbb R}^n_+\times (0,T))}
  &\leq c \Big(\|h\|_{{C}^{\alpha}(\R_+)}+\|g\|_{{C}^{\alpha,\frac{\al}2  }(\Rn\times (0,T))}+\|R'g_n\|_{L^\infty(\R\times (0,T))}
  \\
  &\quad +\|R'
 {h}_n\|_{L^\infty(\R)}+\max\{ T^\frac12 ,  T^{\frac{1}{2}+\frac{\al}{2}} \}\|(u^m\otimes u^m)\|_{{C}^{\alpha,\frac{\al}2 }(\R_+\times (0,T))}
\Big).
 \end{align*}
By the bilinear estimate in Lemma \ref{bilinear}, we have
\[
\|(u^m\otimes u^m)\|_{{C}^{\alpha,\frac{\al}2  }(\R_+\times (0,T))}\leq c
\|u^m\|_{L^\infty(\R_+\times (0,T))}
\|u^m\|_{{C}^{\alpha,\frac{\al}2 }(\R_+\times (0,T))}.
\]
Therefore, we have
\begin{align}
\label{m1}
\| u^{m+1}\|_{{C}^{\al,\frac{\al}2 }({\mathbb R}^n_+\times (0,T))}
 &\leq c_1 \Big(\|h\|_{{C}^{\alpha}(\R_+)}+ \|g\|_{{C}^{\alpha,\frac{\al}2  }(\Rn\times (0,T))}+ \|R'g_n\|_{L^\infty(\Rn\times (0,T))}\\
\notag&\quad+ \|R'{h}_n\|_{L^\infty(\R)}+ \max\{T^\frac12, T^{\frac{1}{2}+\frac{\al}{2}} \} \|u^m\|_{L^\infty(\R_+\times (0,T))}\|u^m\|_{{C}^{\alpha,\frac{\al}2  }(\R_+\times (0,T))} \Big).
 \end{align}

Set \[M_0=\|h\|_{{C}^{\alpha}(\R_+)}
 +\|g\|_{{C}^{\alpha,\frac{\al}2  }(\Rn\times (0,T))}+\|R'g_n\|_{L^\infty(\Rn\times (0,T))}+\|R'{h}_n\|_{L^\infty(\R)}.\]
 Choose $M>2c_1M_0$.
 Then \eqref{m1} implies that
 \[
 \| u^{1}\|_{{C}^{\al,\frac{\al}2 }({\mathbb R}^n_+\times (0,T))}\leq c_1M_0<M,
 \]
 and
under the condition that $\| u^m\|_{{C}^{\al,\frac{\al}2 }({\mathbb R}^n_+\times (0,T))}\leq M$, we have
\[
 \| u^{m+1}\|_{{C}^{\al,\frac{\al}2 }({\mathbb R}^n_+\times (0,T))}\leq c_1M_0+c_1\max\{T^\frac12, T^{\frac{1}{2}+\frac{\al}{2}} \} M^2.
 \]
Choose $0<T \leq    \frac{1}{(2c_1M)^2} $ together with the condition $T\leq 1$. Then by mathematical induction argument we can conclude that
\[
 \| u^{m}\|_{{C}^{\al,\frac{\al}2 }({\mathbb R}^n_+\times (0,T))}\leq M\mbox{ for all }m=1,2,\cdots.
 \]

\subsection{Uniform Convergence}
Let $U^m=u^{m+1}-u^m$ and $P^m=p^{m+1}-p^m$.
Then $U^m$ satisfies the system
\[
\begin{array}{l}\vspace{2mm}
U^m_t - \De U^m + \na P^m =-{\mathbb P}\mbox{div}(u^m\otimes U^{m-1}+U^{m-1}\otimes u^{m-1}), \qquad div \, U^{m} =0, \mbox{ in }
 \R_+\times (0,T),\\
\hspace{30mm}U^{m}|_{t=0}= 0, \qquad  U^{m}|_{x_n =0} =0,
\end{array}
\]

By the result of Theorem \ref{thm1}, we have
\begin{align*}
\|U^m\|_{{C}^{\al,\frac{\al}2 }({\mathbb R}^n_+\times (0,T))}
\leq c \max\{T^\frac12, T^{\frac{1}{2}+\frac{\al}{2}} \} \|(u^m\otimes U^{m-1}+U^{m-1}\otimes u^{m-1})\|_{{C}^{\alpha,\frac{\al}2}(\R_+\times (0,T))}\\
\leq c_2\max\{T^\frac12, T^{\frac{1}{2}+\frac{\al}{2}} \} (\|u^m\|_{{C}^{\alpha,\frac{\al}2}(\R_+\times (0,T))}+\|u^{m-1}\|_{{C}^{\alpha,\frac{\al}2}(\R_+\times (0,T))})\|U^{m-1}\|_{{C}^{\alpha,\frac{\al}2}(\R_+\times (0,T))} .
\end{align*}
Choose $0<T\leq \frac{1}{(4c_2M)^2}  $ together with the condition $T \leq    \frac{1}{(2c_1M)^2}$ and $T\leq 1$. Then, the above estimate leads to the
\begin{equation}
\label{m2}
\|U^m\|_{{C}^{\al,\frac{\al}2 }({\mathbb R}^n_+\times (0,T))}\leq \frac{1}{2}\|U^{m-1}\|_{{C}^{\alpha,\frac{\al}2}(\R_+\times (0,T))}.
\end{equation}
\eqref{m2} implies the infinite series  $\sum_{k=1}^\infty U^k$ converges in ${C}^{\alpha,\frac{\al}2}(\R_+\times (0,T))$.
Observe that $u^m=u^1+\sum_{k=1}^nU^{k}, m=2,3,\cdots.$ Hence $u^m$ converges to $u^1+\sum_{k=1}^\infty U^{k}$ in $ {C}^{\alpha,\frac{\al}2}(\R_+\times (0,T))$.
Set $u:=u^1+\sum_{k=1}^\infty U^{k}.$

\subsection{Existence, regularity and uniqueness}

Let $u$ be the same one constructed by the previous section.
In this section, we will show that $u$ satisfies weak formulation of Navier-Stokes system, that is, $u$ is a weak solution of Navier-Stokes system with appropriate distribution $p$.

Let $\Phi\in C^\infty_{0}(\R_+\times (0,T))$ with $\mbox{div }\Phi=0$.
Observe that
\[
-\int^T_0\int_{\R_+} u^{m+1}\cdot \Delta\Phi dxdt=\int^T_0\int_{\R_+}u^m\cdot (\Phi_t+(u^m\cdot \nabla)\Phi) dxdt.\]
Now send $m$ to the infinity, then, since $u^m\rightarrow u$ in ${C}^{\alpha,\frac{\alpha}{2}}(\R_+\times (0,T))$, we have
\begin{equation}
\label{w1}
\int^T_0\int_{\R_+}u\cdot \Delta \Phi dxdt=\int^T_0\int_{\R_+}u\cdot (\Phi_t+(u\cdot \nabla)\Phi) dxdt.\end{equation}


Note that $u$ can be decomposed by $u=v+V+\nabla \phi+w$, where
\begin{align*}
v(x,t)&=\int_{\R}\Gamma(x-y,t)\tilde{h}(y)dy,\\
V_j(x,t)&=\int^t_{0}\int_{\R}D_{x_k}\Gamma(x-y,t-s)[\delta_{ij}\tilde{u}_{k}\tilde{u}_i+R_iR_j\tilde{u}_{k}\tilde{u}_{i}](y,s)dyds,\\
\phi(x,t)&=2\int_{\Rn}N(x'-y',x_n)\Big(g_n(y',t)-v_n(y',0,t)-V_n(y',0,t)\Big)dy'\\
 w&= \sum_{j=1}^{n-1}\int_0^t \int_{\Rn} K_{ij}( x'-y',x_n,t-s)G_j(y',s) dy'ds,
  \end{align*}
  for $G=(g' -V'|_{x_n=0} -v'|_{x_n=0} - R'(g_n -v_n|_{x_n=0}-V_n|_{x_n=0},0).$
 By Proposition \ref{proheat1},
 $v$ is infinitely differentiable in $(x,t)\in \R_+\times (0,T)$,  by Proposition \ref{propheat2} $V\in C^{\al+1,\frac{\al+1}{2}}(\R_+\times (0,T)),$
 by Proposition \ref{proppoisson2} $\nabla \phi(t)$ is infinitely differentiable in $x \in \R_+$ for each $t>0$, and by Remark  \ref{rem1031}, $w(t)$  is infinitely differentiable in $x\in \R_+$ for each $t>0$, concluding that
 $\nabla u\in L^\infty(K  \times (\delta, T))$  for each $\delta > 0$ and for each compact subset $K$ of $\R_+$.
 Therefore, \eqref{w1}  can be rewritten by
\[
\int^T_0\int_{\R_+}\nabla u:\nabla \Phi dxdt=\int^T_0\int_{\R_+}u\cdot (\Phi_t-(\Phi\cdot \nabla)u) dxdt.\]
This leads to the conclusion that $u$  is a weak solution of the  Navier-Stokes system \eqref{maineq2}.

Let  $ v\in C^{\al,\frac{\al}{2}}(\R_+\times (0,T))$ be  another solution of Naiver-Stokes system \eqref{maineq2} with pressure $q$. Then
 $u-v$ satisfies the system
\begin{align*}
(u-v)_t - \De (u-v) + \na (p-q)& =-\mbox{div}(u\otimes (u-v)+(u-v)\otimes v)\mbox{ in }
 \R_+\times (0,T), \\
 div \, (u-v)& =0,
 \mbox{ in }\R_+\times (0,T),\\
 (u-v)|_{t=0}= 0, &\quad (u-v)|_{x_n =0} =0.
\end{align*}
Applying  Theorem \ref{thm1} to the above Stokes system for $u-v$ and then applying   Lemma \ref{bilinear},
\begin{align*}
\| u-v\|_{{C}^{\al,\frac{\al}2 }({\mathbb R}^n_+\times (0,T_1))}
 \leq c\max\{T_1^\frac12, T_1^{\frac{1}{2}+\frac{\al}{2}}\} \|u\otimes (u-v)+(u-v)\otimes v \|_{{C}^{\alpha,\frac{\al}2 }(\R_+\times (0,T_1))}\\
 \leq c_3\max\{T^\frac12, T^{\frac{1}{2}+\frac{\al}{2}}\}( \|u\|_{L^\infty({\mathbb R}^n_+\times (0,T_1))}+\|v\|_{L^\infty({\mathbb R}^n_+\times (0,T_1))})\| u-v\|_{{C}^{\al,\frac{\al}2 }({\mathbb R}^n_+\times (0,T_1))},\ T_1\leq T.
 \end{align*}
 If we take $T_1\leq \frac{1}{c_3^2(\|u\|_{L^\infty({\mathbb R}^n_+\times (0,T))}+\|v\|_{L^\infty({\mathbb R}^n_+\times (0,T))}+1)^2}$ together with $T_1\leq 1$, then the above inequality leads to the conclusion that
 \[
 \| u-v\|_{{C}^{\al,\frac{\al}2 }({\mathbb R}^n_+\times (0,T_1))}=0\mbox{ that is, }u\equiv v\mbox{ in }\R_+\times (0,T_1).\]
 By the same argument, we can show that
\[
 \| u-v\|_{{C}^{\al,\frac{\al}2 }({\mathbb R}^n_+\times (T_1,2T_1))}=0\mbox{ that is, }u\equiv v\mbox{ in }\R_+\times (T_1,2T_1).\]
After iterating this procedure  finite times, we obtain  the conclusion that $u=v$ in $\R_+\times (0,T)$.


\appendix
\setcounter{equation}{0}

\section{Proof of Proposition \ref{proheat1} }
\setcounter{equation}{0}
\label{appen2}

%
%
%
%
%
%



%
By Young's theorem, we have
\[
\|u(t)\|_{L^\infty(\R)}\leq c\|\Gamma_t\|_{L^1}\|f\|_{L^\infty(\R)}\leq c\|f\|_{L^\infty(\R)},\]  and this gives
the estimate
\begin{equation}
\label{appen.prop1}
\|u\|_{L^\infty(\R\times (0,\infty))}\leq c\|f\|_{L^\infty(\R)}.\end{equation}

Since $D_tu=\Delta_x u$, and $D^2_xu=\Gamma_t*D^2_xf$, again, by Young's theorem we have
\[
\|D_tu(t)\|_{L^\infty(\R)}\leq c\|D_x^2u(t)\|_{L^\infty(\R)}\leq c\|\Gamma_t\|_{L^1}\|D^2_xf\|_{L^\infty(\R)}\leq c\|D^2_xf\|_{L^\infty(\R)},\]
and this gives the estimate
\begin{equation}
\label{appen.prop2}
\|u\|_{\dot{W}^{2,1}_\infty(\R\times (0,\infty))}\leq c\|f\|_{\dot{W}^2_\infty(\R)}.\end{equation}

According to the real interpolation theory,
\[
( L^\infty(\R), \dot{W}_\infty^{2}(\R))_{\frac{\al}2  }=\dot{B}^\al_\infty(\R),\]
and
\[
( L^\infty(\R\times (0,\infty)), \dot{W}_\infty^{2,1}(\R\times (0,\infty)))_{\frac{\al}2  }=\dot{B}^{\al,\frac{\al}{2}}_\infty(\R\times (0,\infty))\]
for $ \quad  0<\al<2$.
 Apply real interpolation theory to \eqref{appen.prop1} and \eqref{appen.prop2}, then  we have the estimate \[
 \|u\|_{\dot{B}^{\al,\frac{\al}{2}}_\infty(\R\times (0,\infty))}\leq c\|f\|_{\dot{B}^\al_\infty(\R)}.\]
The argument can be extended to any $\al>0$.

 The last estimate concerning smoothness comes  easily from the  properties of the heat kernel.

\section{Proof of Proposition \ref{propheat2} }
\label{appen3}
\setcounter{equation}{0}

Let us derive the first estimate of the proposition.
By properties of heat kernel, $\Gamma*_{x,t}{f}\in \dot{B}^{\al+2,\frac{\al}{2}+1}_\infty(\R\times {\mathbb R})$ with
\[
\|\Gamma*_{x,t}{f}\|_{ \dot{B}^{\al+2,\frac{\al}{2}+1}_\infty(\R\times {\mathbb R})}\leq c\|{f}\|_{ \dot{B}^{\al,\frac{\al}{2}}_\infty(\R\times {\mathbb R})},\]
where $*_{x,t}$ means convolution in $(x,t)$ variables.

If $f|_{t=0}=0$, then there is $\tilde{f}\in \dot{B}^{\al, \frac{\al}{2}}_\infty(\R\times {\mathbb R})$ with ${\rm supp} \, \tilde{f}\subset \R\times (0,2T)$ and $\|\tilde{f}\|_{ \dot{B}^{\al,\frac{\al}{2}}_\infty(\R\times {\mathbb R})}\leq c\|f\|_{ \dot{B}^{\al,\frac{\al}{2}}_\infty(\R\times (0,T))}$.
Note that $u(x,t)=D_x \Ga*_{x,t}\tilde{f}$ for $t> 0$.
Since, $\Gamma*_{x,t}\tilde{f}\in \dot{B}^{\al+2,\frac{\al}{2}+1}_\infty(\R\times {\mathbb R})$ with
\[
\|\Gamma*_{x,t}\tilde{f}\|_{ \dot{B}^{\al+2,\frac{\al}{2}+1}_\infty(\R\times {\mathbb R})}\leq c\|\tilde{f}\|_{ \dot{B}^{\al,\frac{\al}{2}}_\infty(\R\times {\mathbb R})},\]
we have
\begin{align}
\|u\|_{ \dot{B}^{\al+1,\frac{\al}{2}+\frac{1}{2}}_\infty(\R\times (0,T))}
\leq c\|{f}\|_{ \dot{B}^{\al,\frac{\al}{2}}_\infty(\R\times (0,T))}.
\end{align}

If $f|_{t=0}\neq 0$, then  let $F(s)=f(s)-\Gamma_s*(f|_{t=0})$ and $U=\int^t_0\int_{\R}D_y\Gamma_{t-s}*F(s) dyds$. Then
\begin{align*}
\|U\|_{ \dot{B}^{\al+1,\frac{\al}{2}+\frac{1}{2}}_\infty(\R\times (0,T))}
\leq c\|{F}\|_{ \dot{B}^{\al,\frac{\al}{2}}_\infty(\R\times(0,T))}.
\end{align*}
Note that
$\int^t_0\int_{\R}D_y\Gamma_{t-s}*\Big( \Gamma_s*(f|_{t=0})\Big) ds=t\int_{\R}D_y\Gamma_{t}*(f|_{t=0}) dy,$ and
$D_x\Big(t\int_{\R}D_y\Gamma_{t}*(f|_{t=0}) dy\Big)\sim \int_{\R}\ \Gamma_{t}*(f|_{t=0}) dy.$
By the same reasoning as for the proof of Proposition \ref{proheat1}, we can show that $t\int_{\R}D_y\Gamma_{t}*(f|_{t=0}) dy\in \dot{B}^{\al+1,\frac{\al+1}{2}}_\infty(\R\times(0,T))$, $0<\al$ with
\begin{align*}
\|t\int_{\R}D_y\Gamma_{t}*(f|_{t=0}) dy\|_{ \dot{B}^{\al+1,\frac{\al}{2}+\frac{1}{2}}_\infty(\R\times (0,T))}
\leq c\|f|_{t=0}\|_{ \dot{B}^{\al,\frac{\al}{2}}_\infty(\R)}\leq c\|f\|_{ \dot{B}^{\al,\frac{\al}{2}}_\infty(\R\times(0,T))}.
\end{align*}
Combining the above two estimates we conclude that
\begin{align}
\notag\|u\|_{ \dot{B}^{\al+1,\frac{\al}{2}+\frac{1}{2}}_\infty(\R\times (0,T))}&\leq \|U\|_{ \dot{B}^{\al+1,\frac{\al}{2}+\frac{1}{2}}_\infty(\R\times (0,T))}\\
&+\|t\int_{\R}D_y\Gamma_{t}*(f|_{t=0}) dy\|_{ \dot{B}_\infty^{\al+1,\frac{\al}{2}+\frac{1}{2}}(\R\times (0,T))}
\leq c\|f\|_{ \dot{B}^{\al,\frac{\al}{2}}_\infty(\R\times(0,T))}
\end{align}

It is well known that $D_x\Ga(t)\in {\mathcal H}^1(\R)$, where ${\mathcal H}^1(\R)$ denotes  Hardy space.
Since  $\|D_x \Ga(t)\|_{{\mathcal H}^1(\R)} \leq ct^{-\frac12}$,
we have
\begin{align}
\label{h5}
\|u(t)\|_{L^\infty(\R)} & \leq c \int^t_0\int_{\R}\| D_y \Ga(\cdot-y,t-s)\|_{{\mathcal H}^1(\R)} \|f(s)\|_{BMO(\R)}ds
\leq cT^{\frac{1}{2}}\|f\|_{L^\infty( (0,T);BMO(\R))}.
\end{align}
This gives the third estimate of the proposition.

Finally, we will derive the second estimate of the proposition.
Since $ D^2_xu=\int^t_0D_y\Gamma_{t-s}* D^2_y f(s) ds,$
by Young's Theorem we have
\begin{align*}
\|D_x^2u(t)\|_{L^\infty(\R)} \leq c \int^t_0\| D_y \Ga(\cdot-y,t-s)\|_{L^1(\R)} \|D_y^2f(s)\|_{L^\infty(\R)}ds
\leq cT^{\frac{1}{2}}\|D^2_yf\|_{L^\infty(\R\times (0,T))}.
\end{align*}
Since
$D_tu=\int^t_0D_y\Gamma(t-s)*_xD_tf(s) ds,$
\begin{align*}
\|D_tu(t)\|_{L^\infty(\R)}  \leq c \int^t_0\| D_y \Ga(\cdot-y,t-s)\|_{L^1(\R)} \|D_tf(s)\|_{L^\infty(\R)}ds
\leq cT^{\frac{1}{2}}\|D_tf\|_{L^\infty(\R\times (0,T))}.
\end{align*}
This gives the estimate
\begin{align}
\label{h55}
\|u\|_{\dot{W}^{2,1}_\infty(\R\times (0,T))}
\leq cT^{\frac{1}{2}}\|f\|_{\dot W^{2,1}_\infty(\R\times (0,T))}.
\end{align}
Apply real interpolation theory to \eqref{h5} and \eqref{h55}, we have the estimate
\begin{align*}
\|u\|_{\dot{B}^{\al,\frac{\al}{2}}_\infty(\R\times (0,T))}
\leq cT^{\frac{1}{2}}\|f\|_{\dot B_\infty^{\al,\frac{\al}{2}}(\R\times (0,T))}, \ 0<\al<2.
\end{align*}
The argument can be extended to any $\al>0$.

\section{Proof of Proposition \ref{propriesz}.}
\label{appen-riesz}
The first   estimates in Proposition \ref{propriesz} are  well known properties of the singular integral operator(see  \cite{giga}, \cite{sawada}  and  \cite{St}). Hence we have only to prove the second  estimates.
By the similar argument as  in  \cite{giga}, it holds \[
\|Rf\|_{\dot{B}^{\al,\frac{\al}{2}}_\infty(\R\times {\mathbb R})}\leq c\|f\|_{\dot{B}^{\al,\frac{\al}{2}}_\infty(\R\times {\mathbb R})},\ \al\in {\mathbb R}, \mbox{ for any }f\in \dot{B}^{\al,\frac{\al}{2}}_\infty(\R\times {\mathbb R}) .
\]
If $f\in \dot{B}^{\al,\frac{\al}{2}}_\infty(\R\times (0,T))$ with $f|_{t=0}=0$, then there is $\tilde{f}\in \dot{B}^{\al,\frac{\al}{2}}_\infty(\R\times {\mathbb R})$ extension of $f$ with
\[
\|\tilde{f}\|_{\dot{B}^{\al,\frac{\al}{2}}_\infty(\R\times {\mathbb R})}\leq c\|f\|_{\dot{B}^{\al,\frac{\al}{2}}_\infty(\R\times (0,T))},\]
hence
 \[
\|Rf\|_{\dot{B}^{\al,\frac{\al}{2}}_\infty(\R\times (0,T))}\leq c\|f\|_{\dot{B}^{\al,\frac{\al}{2}}_\infty(\R\times (0,T))}\ \al\in {\mathbb R},\mbox{ for any }f\in \dot{B}^{\al,\frac{\al}{2}}_\infty(\R\times (0,T))\mbox{ with }f|_{t=0}=0.
\]

Now let us consider $f\in\dot{B}^{\al,\frac{\al}{2}}_\infty(\R\times (0,T))$ with $f|_{t=0}\neq 0$.
Let $F=f-\Ga_t*(f|_{t=0})$, then $F|_{t=0}=0$. Hence
 \[
\|RF\|_{\dot{B}^{\al,\frac{\al}{2}}_\infty(\R\times (0,T))}\leq c\|F\|_{\dot{B}^{\al,\frac{\al}{2}}_\infty(\R\times (0,T))},\ \al\in {\mathbb R}.
\]
Note that
$R\Big(\Ga_t*(f|_{t=0})\Big)=\Ga_t*\Big(R(f|_{t=0})\Big)$, and
\[
\|\Ga_t*\Big(R(f|_{t=0})\Big)\|_{\dot{B}^{\al,\frac{\al}{2}}_\infty(\R\times (0,T))}\leq c\|R(f|_{t=0})\|_{\dot{B}^\al_\infty(\R)}\leq c\|f|_{t=0}\|_{\dot{B}^\al_\infty(\R)}
\leq  c\|f\|_{\dot{B}^{\al,\frac{\al}{2}}_\infty(\R\times (0,T))}.\]
Here the first inequality and the last inequality hold for $0<\al$.
Therefore we conclude that
 \[
\|Rf\|_{\dot{B}^{\al,\frac{\al}{2}}_\infty(\R\times (0,T))}\leq c\|f\|_{\dot{B}^{\al,\frac{\al}{2}}_\infty(\R\times (0,T))}\mbox{ for any }f\in \dot{B}^{\al,\frac{\al}{2}}_\infty(\R\times (0,T)),\ \al>0 .
\]

\section{Proof of proposition \ref{proppoisson2}.}

\label{appen-poisson}

The first estimate of Proposition \ref{proppoisson2}  is  well known property of Poisson operator (see \cite{St}).
Hence we have only to prove the second two estimates.
By the first estimate 
\begin{equation}
\label{appen.prop3}
\|Pf\|_{L^\infty(\R_+\times (0,T))}\leq c\|f\|_{L^\infty(\Rn\times (0,T))}
.\end{equation}
Since $D^2_{x'}Pf(t)=P(D^2_{x'}f)$, $D_{x_n}^2Pf=-\Delta_{x'}Pf$ and $D_tPf=P(D_tf)$, we have
\begin{align*}
\|D_{x_n}^2Pf\|_{L^\infty(\R_+\times (0,T))}\leq \|D_{x'}^2Pf\|_{L^\infty(\R_+\times (0,T))}\leq c\|D_{x'}f\|_{L^\infty(\Rn\times (0,T))}
,\\
\|D_tPf\|_{L^\infty(\R_+\times (0,T))}\leq c\|D_tf\|_{L^\infty(\Rn\times (0,T))}.
\end{align*}
This gives the estimate
\begin{equation}
\label{appen.prop4}
\|Pf\|_{\dot{W}^{2,1}_\infty(\R\times (0,T))}\leq c\|f\|_{\dot{W}^{2,1}_\infty(\Rn\times (0,T))}.\end{equation}
Apply real interpolation theory to \eqref{appen.prop3} and \eqref{appen.prop4}, we have the estimate
\begin{equation*}
\|Pf\|_{\dot{B}^{\al,\frac{\al}{2}}_\infty(\R\times (0,T))}\leq c\|f\|_{\dot{B}^{\al,\frac{\al}{2}}_\infty(\Rn\times (0,T))},  \ 0<\al<2.\end{equation*}
The argument can  be extended to any $\al>0.$

The last estimate concerning smoothness comes  easily from the properties of the Poisson kernel.

\section*{Acknowledgements}
This research was supported by Basic Science Research Program through the National Research Foundation of Korea(NRF) funded by the Ministry of Education(2014R1A1A3A04049515).

\end{document}